\documentclass[12pt,reqno]{amsart}
\usepackage{amsmath,amssymb,amsfonts,amscd,latexsym,amsthm,mathrsfs,verbatim,comment,cite}
\usepackage{hyperref}
\newtheorem{theorem}{Theorem}[section]
\newtheorem{conjecture}[theorem]{Conjecture}

\newtheorem{corollary}[theorem]{Corollary}

\numberwithin{equation}{section}

\begin{document}

\title[]{A family of $q$-congruences modulo the square of a cyclotomic polynomial}

%\date{15 December 2019}

\author{Victor J. W. Guo}
\address{School of Mathematics and Statistics, Huaiyin Normal University, Huai'an 223300, Jiangsu, People's Republic of China}
\email{jwguo@hytc.edu.cn}

\thanks{The author was partially supported by the National Natural Science Foundation of China (grant 11771175).}


\begin{abstract}Using Watson's terminating $_8\phi_7$ transformation formula,
we prove a family of $q$-congruences modulo the square of a cyclotomic polynomial,
which were originally conjectured by the author and Zudilin [J. Math. Anal. Appl. 475 (2019), 1636--646].
As an application, we deduce  two supercongruences modulo $p^4$ ($p$ is an odd prime)
and their $q$-analogues. This also partially confirms a special case of Swisher's (H.3) conjecture.

\end{abstract}

\maketitle

\section{Introduction}\label{sec1}
In 1997, Van Hamme \cite[(H.2)]{Hamme} proved the following
supercongruence: for any prime $p\equiv 3\pmod 4$,
\begin{equation}
\sum_{k=0}^{(p-1)/2} \frac{(\frac{1}{2})_k^3}{k!^3}
\equiv 0\pmod{p^2},   \label{eq:h2}
\end{equation}
where $(a)_n=a(a+1)\cdots(a+n-1)$ is the rising factorial. It is easy to see that
\eqref{eq:h2} is also true when the sum is over $k$ from $0$ to $p-1$, since
$(1/2)_k^3/k!^3\equiv 0\pmod{p^3}$ for $(p-1)/2<k\leqslant (p-1)/2$.
Nowadays various generalizations of \eqref{eq:h2} can be found in \cite{GZ15,Liu0,LR,GuoZu2,GuoZu3,Sun2,Sun}.
For example, Liu \cite{Liu0} proved that, for any prime $p\equiv 3\pmod{4}$ and positive integer $m$,
\begin{equation}
\sum_{k=0}^{mp-1} \frac{(\frac{1}{2})_k^3}{k!^3}
\equiv 0\pmod{p^2}.   \label{eq:liu}
\end{equation}

The first purpose of this note is to prove the following $q$-analogue of \eqref{eq:liu}, which was originally conjectured by the
author and Zudilin \cite[Conjecture 2]{GuoZu2}.
\begin{theorem}\label{main-1}
Let $m$ and $n$ be positive integers with $n\equiv 3\pmod{4}$. Then
\begin{align}
\sum_{k=0}^{mn-1}\frac{(1+q^{4k+1})(q^2;q^4)_k^3}{(1+q)(q^4;q^4)_k^3} q^k
&\equiv 0\pmod{\Phi_n(q)^2},
\label{eq:guozu-1}\\[5pt]
\sum_{k=0}^{mn+(n-1)/2}\frac{(1+q^{4k+1})(q^2;q^4)_k^3}{(1+q)(q^4;q^4)_k^3} q^k
&\equiv 0\pmod{\Phi_n(q)^2}. \label{eq:guozu-2}
\end{align}
\end{theorem}

Here and in what follows, the {\em $q$-shifted factorial} is defined by $(a;q)_0=1$ and $(a;q)_n=(1-a)(1-aq)\cdots (1-aq^{n-1})$ for
$n\geqslant 1$, and the $n$-th {\em cyclotomic polynomial} $\Phi_n(q)$ is defined as
\begin{align*}
\Phi_n(q)=\prod_{\substack{1\leqslant k\leqslant n\\ \gcd(n,k)=1}}(q-\zeta^k),
\end{align*}
where $\zeta$ is an $n$-th primitive root of unity. Moreover, the {\em $q$-integer} is given by
$[n]=[n]_q=1+q+\cdots+q^{n-1}$.

The $r=1$ case of \eqref{eq:guozu-1} was first conjectured by the author and Zudilin \cite[Conjecture 4.13]{GuoZu} and has already
been proved by themselves in a recent paper \cite{GuoZu3}.
For some other recent progress on $q$-congruences, the reader may consult
\cite{GG,Guo-m3,GL,GS2,GS3,GZ15,GuoZu2,GuoZu3,NP}.

In 2016, Swisher \cite[(H.3) with $r=2$]{Swisher} conjectured that, for primes $p\equiv 3\pmod{4}$ and $p>3$,
\begin{equation}
\sum_{k=0}^{(p^2-1)/2} \frac{(\frac{1}{2})_k^3}{k!^3}
\equiv p^2\pmod{p^5},   \label{eq:h3}
\end{equation}
The second purpose of this note is to the following $q$-congruences related to \eqref{eq:h3} modulo $p^4$.
\begin{theorem}\label{main-2}
Let $n\equiv 3\pmod{4}$ be a positive integer. Then, modulo $\Phi_n(q)^2\Phi_{n^2}(q)^2$, we have
\begin{align}
\sum_{k=0}^{(n^2-1)/2}\frac{(1+q^{4k+1})(q^2;q^4)_k^3}{(1+q)(q^4;q^4)_k^3}q^k
\equiv \dfrac{[n^2]_{q^2}(q^3;q^4)_{(n^2-1)/2}} {(q^5;q^4)_{(n^2-1)/2}} q^{(1-n^2)/2}, \label{eq:main-2-1}\\[5pt]
\sum_{k=0}^{n^2-1}\frac{(1+q^{4k+1})(q^2;q^4)_k^3}{(1+q)(q^4;q^4)_k^3}q^k
\equiv \dfrac{[n^2]_{q^2}(q^3;q^4)_{(n^2-1)/2}} {(q^5;q^4)_{(n^2-1)/2}} q^{(1-n^2)/2}.  \label{eq:main-2-2}
\end{align}
\end{theorem}

Let $n=p\equiv 3\pmod{4}$ be a prime and taking $q\to1$ in Theorem \ref{main-2}. Then $\Phi_p(1)=\Phi_{p^2}(1)=p$, and
\begin{align*}
\lim_{q\to 1}\dfrac{(q^{3};q^4)_{(p^2-1)/2}} {(q^5;q^4)_{(p^2-1)/2}}
=\prod_{k=1}^{(p^2-1)/2}\frac{4k-1}{4k+1}=\frac{(\frac{3}{4})_{(p^2-1)/2}}{(\frac{5}{4})_{(p^2-1)/2}}.
\end{align*}
and we obtain the following conclusion.

\begin{corollary}Let $p\equiv 3\pmod{4}$ be a prime. Then
\begin{align}
\sum_{k=0}^{(p^2-1)/2} \frac{(\frac{1}{2})_k^3}{k!^3}
&\equiv p^2\frac{(\frac{3}{4})_{(p^2-1)/2}}{(\frac{5}{4})_{(p^2-1)/2}} \pmod{p^4},   \label{eq:h3-1} \\[5pt]
\sum_{k=0}^{p^2-1} \frac{(\frac{1}{2})_k^3}{k!^3}
&\equiv p^2\frac{(\frac{3}{4})_{(p^2-1)/2}}{(\frac{5}{4})_{(p^2-1)/2}} \pmod{p^4}.   \label{eq:h3-2}
\end{align}
\end{corollary}

Comparing \eqref{eq:h3} and \eqref{eq:h3-1}, we would like to propose the following conjecture.
\begin{conjecture}\label{conj:first}
Let $p\equiv 3\pmod{4}$ be a prime and $r$ a positive integer. Then
\begin{align*}
\prod_{k=1}^{(p^{2r}-1)/2}\frac{4k-1}{4k+1}\equiv 1\pmod{p^2}.
\end{align*}
\end{conjecture}
Note that the $r=1$
case is equivalent to say that \eqref{eq:h3} is true modulo $p^4$.

%%%%%%%%%%%%%%%%%%%%%%%%%%%%%%%%%%%%%%%%%%%%%%%%%%%%%%%%%%%%%%%%%%%%%%%%%%%%%%%%%%%%%%%%%%%%%%%%%%%%%%%%%%%%%%%%%%%%%%%%%%%%%%%%%%%%%%%%%%%%%%%%%%%%%%%%%%%%%%%%%%%%%%%%%%%%%%
\section{Proof of Theorem \ref{main-1}}
We need to use Watson's terminating $_8\phi_7$ transformation formula (see \cite[Appendix (III.18)]{GR}):
\cite[Section 1]{GR}:
\begin{align}
& _{8}\phi_{7}\!\left[\begin{array}{cccccccc}
a,& qa^{\frac{1}{2}},& -qa^{\frac{1}{2}}, & b,    & c,    & d,    & e,    & q^{-n} \\
  & a^{\frac{1}{2}}, & -a^{\frac{1}{2}},  & aq/b, & aq/c, & aq/d, & aq/e, & aq^{n+1}
\end{array};q,\, \frac{a^2q^{n+2}}{bcde}
\right] \notag\\[5pt]
&\quad =\frac{(aq;q)_n(aq/de;q)_n}
{(aq/d;q)_n(aq/e;q)_n}
\,{}_{4}\phi_{3}\!\left[\begin{array}{c}
aq/bc,\ d,\ e,\ q^{-n} \\
aq/b,\, aq/c,\, deq^{-n}/a
\end{array};q,\, q
\right],  \label{eq:8phi7}
\end{align}
where the basic hypergeometric ${}_{m+1}\phi_r$ series with $r$ upper parameters
$a_1,\dots,a_{m+1}$, $s$ lower parameters $b_1,\dots,b_r$, base $q$ and
argument $z$ is defined as
\begin{equation*}
{}_{m+1}\phi_{r}\!\left[\begin{matrix}
a_1,a_2,\dots,a_{m+1}\\b_1,\dots,b_r
\end{matrix};q,z\right]:=\sum_{k=0}^\infty
\frac{(a_1;q)_k(a_2;q)_k \dots(a_{m+1};q)_k}{(q;q)_k(b_1;q)_k\cdots(b_r;q)_k} z^k.
\end{equation*}

The left-hand side of \eqref{eq:guozu-2} with $r\geqslant 0$ can be written as the following terminating $_8\phi_7$ series:
\begin{equation}
_{8}\phi_{7}\!\left[\begin{array}{cccccccc}
q^2,& q^5, & -q^5, & q^2,    & q,      & q^2,      & q^{4+(4m+2)n},  & q^{2-(4m+2)n} \\
    & q,   & -q,   & q^4,    & q^5,    & q^4,      & q^{2-(4m+2)n},  & q^{4+(4m+2)n}
\end{array};q^4,\, q
\right].   \label{eq:8phi7-1}
\end{equation}
By Watson's transformation formula  \eqref{eq:8phi7} with $q\mapsto q^4$, $a=b=d=q^2$,  $c=q$, $e=q^{4+(4m+2)n}$, and $n\mapsto mn+(n-1)/2$,
we see that \eqref{eq:8phi7-1} is equal to
\begin{equation}
\frac{(q^6;q^4)_{mn+(n-1)/2}(q^{-(4m+2)n};q^4)_{mn+(n-1)/2}}
{(q^4;q^4)_{mn+(n-1)/2}(q^{2-(4m+2)n};q^4)_{mn+(n-1)/2}}
\,{}_{4}\phi_{3}\!\left[\begin{array}{c}
q^3,\ q^2,\, q^{4+(4m+2)n},\ q^{2-(4m+2)n}  \\
q^4,\quad q^5,\quad q^6
\end{array};q^4,\, q^4
\right].  \label{eq:4phi3-2}
\end{equation}
It is not difficult to see that there are exactly $m+1$ factors of the form $1-q^{an}$ ($a$ is an integer) among
the $mn+(n-1)/2$ factors of  $(q^6;q^4)_{mn+(n-1)/2}$. So are $(q^{-(4m+2)n};q^4)_{mn+(n-1)/2}$.
But there are only $r$ factors of the form $1-q^{an}$ ($a$ is an integer) in each of
$(q^4;q^4)_{mn+(n-1)/2}$ and $(q^{2-(4m+2)n};q^4)_{mn+(n-1)/2}$. Since $\Phi_n(q)$ is a factor of $1-q^N$
if and only if $n$ divides $N$, we conclude that the fraction before the $_4\phi_3$
series is congruent to $0$ modulo $\Phi_n(q)^2$. Moreover, for any integer $x$, let $f_n(x)$ be the least non-negative
integer $k$ such that $(q^x,q^4)_k\equiv 0$ modulo $\Phi_n(q)$. Since $n\equiv 3\pmod{4}$,
we have $f_n(2)=(n+1)/2$, $f_n(3)=(n+1)/4$, $f_n(4)=n$, $f_n(5)=(3n-1)/4$, and $f_n(6)=(n-1)/2$.
It follows that the denominator of the reduced form of the $k$-th summand in the $_4\phi_3$ summation
$$
\frac{(q^3;q^4)_k (q^2;q^4)_k(q^{4+(4m+2)n};q^4)_k(q^{2-(4m+2)n};q^4)_k}
{(q^4;q^4)_k^2(q^5;q^4)_k(q^6;q^4)_k}q^{4k}
$$
is always relatively prime to $\Phi_n(q)$ for any non-negative integer $k$.
This proves that \eqref{eq:4phi3-2} (i.e. \eqref{eq:8phi7-1}) is congruent to $0$ modulo $\Phi_n(q)^2$,
thus establishing \eqref{eq:guozu-2} for $m\geqslant 0$.

It is easy to see that $(q^2;q^4)_k^3/(q^4;q^4)_k^3$ is congruent to $0$ modulo $\Phi_n(q)^3$ for
$mn+(n-1)/2\leqslant(m+1)n-1$. Therefore, the $q$-congruence \eqref{eq:guozu-1} with $m\mapsto m+1$
follows from \eqref{eq:guozu-2}.

%%%%%%%%%%%%%%%%%%%%%%%%%%%%%%%%%%%%%%%%%%%%%%%%%%%%%%%%%%%%%%%%%%%%%%%%%%%%%%%%%%%%%%%%%%%%%%%%%%%%%%%%%%%%%%%%%%%%%%%%%%%%%%%%%%%%%%%%%%%%%%%%%%%%%%%%%%%%%%%%%%%%%%%%%%%%%%
\section{Proof of Theorem \ref{main-2}}
The author and Zudilin \cite[Theorem 1.1]{GuoZu3} proved that, for any positive odd integer $n$,
\begin{equation}
\sum_{k=0}^{(n-1)/2}\frac{(1+q^{4k+1})\,(q^2;q^4)_k^3}{(1+q)\,(q^4;q^4)_k^3}\,q^{k}
\equiv\dfrac{[n]_{q^2}(q^3;q^4)_{(n-1)/2}}{(q^5;q^4)_{(n-1)/2}}\,q^{(1-n)/2}
\pmod{\Phi_n(q)^2},  \label{eq:guozu-3}
\end{equation}
which is also true when the sum on the left-hand side of \eqref{eq:guozu-3} is over $k$ from $0$ to $n-1$.
Replacing $n$ by $n^2$ in \eqref{eq:guozu-3} and its equivalent form, we see that the $q$-congruences \eqref{eq:main-2-1} and \eqref{eq:main-2-2}
hold modulo $\Phi_{n^2}(q)^2$.

It is easy to see that, for $n\equiv 3\pmod{4}$,
$$
\dfrac{[n^2]_{q^2}(q^3;q^4)_{(n^2-1)/2}} {(q^5;q^4)_{(n^2-1)/2}} q^{(1-n^2)/2}\equiv 0\pmod{\Phi_n(q)^2}
$$
because $[n^2]_{q^2}=(1-q^{n^2})/(1-q^2)$ is divisible by $\Phi_n(q)$, and $(q^3;q^4)_{(n^2-1)/2}$ contains $(n+1)/2$
factor of the for $1-q^{an}$ ($a$ is an integer), while $(q^5;q^4)_{(n^2-1)/2}$ only has $(n-1)/2$ such factors.
Meanwhile, by Theorem \ref{main-1}, the left-hand sides of \eqref{eq:main-2-1} and \eqref{eq:main-2-2}
are both congruent to $0$ modulo $\Phi_n(q)^2$ since $(n^2-1)/2=(n-1)n/2+(n-1)/2$.
This proves that the $q$-congruences \eqref{eq:main-2-1} and \eqref{eq:main-2-2} also
hold modulo $\Phi_n(q)^2$. Since the polynomials $\Phi_n(q)$ and $\Phi_{n^2}(q)$ are
relatively prime, we finish the proof of the theorem.

%%%%%%%%%%%%%%%%%%%%%%%%%%%%%%%%%%%%%%%%%%%%%%%%%%%%%%%%%%%%%%%%%%%%%%%%%%%%%%%%%%%%%%%%%%%%%%%%%%%%%%%%%%%%%%%%%%%%%%%%%%%%%%%%%%%%%%%%%%%%%%%%%%%%%%%%%%%%%%%%%%%%%%%%%%%%%%%%%%%%%%%%%%
\section{Discussion}
Swisher's (H.3) conjecture also indicates that, for positive integer $r$ and primes $p\equiv 3\pmod{4}$ with $p>3$, we have
\begin{equation}
\sum_{k=0}^{(p^{2r}-1)/2} \frac{(\frac{1}{2})_k^3}{k!^3}
\equiv p^{2r}\pmod{p^{2r+3}}.  \label{eq:h3-new}
\end{equation}
Motivated by \eqref{eq:h3-new}, we shall give the following generalization of Theorem \ref{main-2}.
\begin{theorem}\label{main-3}
Let $n$ and $r$ be positive integers with $n\equiv 3\pmod{4}$. Then, modulo $\Phi_{n^{2r}}(q)^2\prod_{j=1}^{r}\Phi_{n^{2j-1}}(q)^2$, we have
\begin{align}
\sum_{k=0}^{(n^{2r}-1)/2}\frac{(1+q^{4k+1})(q^2;q^4)_k^3}{(1+q)(q^4;q^4)_k^3}q^k
\equiv \dfrac{[n^{2r}]_{q^2}(q^3;q^4)_{(n^{2r}-1)/2}} {(q^5;q^4)_{(n^{2r}-1)/2}} q^{(1-n^{2r})/2}, \label{eq:main-3-1}\\[5pt]
\sum_{k=0}^{n^{2r}-1}\frac{(1+q^{4k+1})(q^2;q^4)_k^3}{(1+q)(q^4;q^4)_k^3}q^k
\equiv \dfrac{[n^{2r}]_{q^2}(q^3;q^4)_{(n^{2r}-1)/2}} {(q^5;q^4)_{(n^{2r}-1)/2}} q^{(1-n^{2r})/2}.    \label{eq:main-3-2}
\end{align}
\end{theorem}
\begin{proof}
Replacing $n$ by $n^{2r}$ in \eqref{eq:guozu-3} and its equivalent form, we see that \eqref{eq:main-3-1} and \eqref{eq:main-3-2}
are true modulo $\Phi_{n^{2r}}(q)^2$.  Similarly as before, we can show that
\begin{align*}
\dfrac{[n^{2r}]_{q^2}(q^3;q^4)_{(n^{2r}-1)/2}} {(q^5;q^4)_{(n^{2r}-1)/2}} q^{(1-n^{2r})/2}
\equiv 0\pmod{\prod_{j=1}^{r}\Phi_{n^{2j-1}}(q)^2}.
\end{align*}
Further, by Theorem \ref{main-1}, we can easily deduce that
the left-hand sides of \eqref{eq:main-3-1} and \eqref{eq:main-3-2}
are also congruent to $0$ modulo $\prod_{j=1}^{r}\Phi_{n^{2j-1}}(q)^2$.
\end{proof}

Let $n=p\equiv 3\pmod{4}$ be a prime and taking $q\to1$ in Theorem \ref{main-3}, we are led to the following result.
\begin{corollary}Let $p\equiv 3\pmod{4}$ be a prime and let $r\geqslant 1$. Then
\begin{align}
\sum_{k=0}^{(p^{2r}-1)/2} \frac{(\frac{1}{2})_k^3}{k!^3}
&\equiv p^{2r}\frac{(\frac{3}{4})_{(p^{2r}-1)/2}}{(\frac{5}{4})_{(p^{2r}-1)/2}} \pmod{p^{2r+2}}, \label{eq:hnew-1}    \\[5pt]
\sum_{k=0}^{p^{2r}-1} \frac{(\frac{1}{2})_k^3}{k!^3}
&\equiv p^{2r}\frac{(\frac{3}{4})_{(p^{2r}-1)/2}}{(\frac{5}{4})_{(p^{2r}-1)/2}} \pmod{p^{2r+2}}.  \label{eq:hnew-2}
\end{align}
\end{corollary}
If Conjecture \ref{conj:first} is true, then \eqref{eq:hnew-1} means that \eqref{eq:h3-new}
holds modulo $p^{2r+2}$ for any odd prime $p$.

It is known that $q$-analogues of supercongruences are usually not unique. See, for example, \cite{GG}.
The author and Zudilin \cite[Conjecture 1]{GuoZu2} also gave another $q$-analogue of \eqref{eq:liu},  which still remains open.
\begin{conjecture}[Guo and Zudilin] \label{conj:fina-1}
Let $m$ and $n$ be positive integers with $n\equiv 3\pmod{4}$. Then
\begin{align}
\sum_{k=0}^{mn-1}\frac{(q;q^2)_k^2(q^2;q^4)_k}{(q^2;q^2)_k^2(q^4;q^4)_k}\,q^{2k}
&\equiv 0 \pmod{\Phi_n(q)^2},  \label{eq:final}
\\
\sum_{k=0}^{mn+(n-1)/2}\frac{(q;q^2)_k^2(q^2;q^4)_k}{(q^2;q^2)_k^2(q^4;q^4)_k}\,q^{2k}
&\equiv 0 \pmod{\Phi_n(q)^2}.  \notag
\end{align}
\end{conjecture}

The author and Zudilin \cite[Theorem 2]{GuoZu2} themselves have proved \eqref{eq:final} for the $m=1$ case.
Motivated by Conjecture  \ref{conj:fina-1}, we would like to give the following new conjectural $q$-analogues of
\eqref{eq:h3-1} and \eqref{eq:h3-2}.

\begin{conjecture}
Let $n\equiv 3\pmod{4}$ be a positive integer. Then, modulo $\Phi_n(q)^2\Phi_{n^2}(q)^2$, we have
\begin{align*}
\sum_{k=0}^{(n^2-1)/2}\frac{(q;q^2)_k^2(q^2;q^4)_k}{(q^2;q^2)_k^2(q^4;q^4)_k}q^{2k}
&\equiv \dfrac{[n^2](q^3;q^4)_{(n^2-1)/2}} {(q^5;q^4)_{(n^2-1)/2}}, \\[5pt]
\sum_{k=0}^{n^2-1}\frac{(q;q^2)_k^2(q^2;q^4)_k}{(q^2;q^2)_k^2(q^4;q^4)_k}q^{2k}
&\equiv \dfrac{[n^2](q^3;q^4)_{(n^2-1)/2}} {(q^5;q^4)_{(n^2-1)/2}}.
\end{align*}
\end{conjecture}
There are similar new such $q$-analogues of \eqref{eq:hnew-1} and \eqref{eq:hnew-2}.
We omit them here and leave space for the reader's imagination.

\end{document}